\documentclass[10pt]{amsart}
\usepackage{amsmath,amssymb,amscd,amsfonts,amsthm,mathrsfs,eucal,dsfont,verbatim}
\usepackage{lscape,enumerate}
\usepackage{natbib}
\usepackage{setspace}
\usepackage{marginnote}

\usepackage{geometry}
\usepackage[colorlinks,citecolor=blue]{hyperref}
\usepackage{color}

\newtheorem{theorem}{Theorem}
\newtheorem{remark}{Remark}
\numberwithin{equation}{section}

\title{How much we gain by surplus-dependent premiums
- asymptotic analysis of ruin probability}

\author[C. Constantinescu]{Corina Constantinescu}
\address{University of Liverpool, Liverpool, UK}
\email{C.Constantinescu@liverpool.ac.uk}

\author[Z. Palmowski]{Zbigniew Palmowski}
\address{Wroc\l aw University of Science and Technology, Poland}
\email{zbigniew.palmowski@pwr.edu.pl}

\author[J. Wang]{Jing Wang}
\address{University of Liverpool, Liverpool, UK}
\email{Jing.Wang3@liverpool.ac.uk}

\thanks{This work is partially supported by Polish National Science Centre Grant No. 2018/29/B/ST1/00756, 2019-2022}

\begin{document}

\begin{abstract}
In this paper, we build on the techniques developed in \citet{exactasymptotic},
to generate initial-boundary value problems
for ruin probabilities
of surplus-dependent premium risk processes, under a renewal case scenario, Erlang (2) claim arrivals, 
and an exponential claims scenario, Erlang (2) claim sizes.
Applying the approximation theory of solutions of linear
ordinary differential equations developed in \citet{Fed1993}, we derive the asymptotics of the ruin probabilities when the initial reserve tends to infinity.
When considering premiums that are {\it linearly} dependent on reserves, representing for instance returns on risk-free investments of the insurance capital, we firstly derive explicit formulas for the ruin probabilities, from which we can easily determine their
asymptotics, only to match the ones obtained for general premiums dependent on reserves. We compare them with the asymptotics of the equivalent ruin probabilities when the premium rate is fixed over time, to
measure the gain generated by this additional mechanism of binding the premium rates with the amount of reserve own by the insurance company.\vspace{3mm}

\noindent {\sc Keywords.} ruin probability $\star$ premiums dependent on reserves $\star$ risk process $\star$ Erlang distribution
\end{abstract}

\maketitle

\pagestyle{myheadings} \markboth{\sc C.\ Constantinescu--- Z.\ Palmowski  --- J.\ Wang} {\sc
Ruin probabilities and surplus-dependent premiums}

\section{Introduction}
Insurance companies maintain solvency via careful design of premiums rates. The premiums rates are primarily based on the claims history and carefully adjusted to evolving factors such as the number of customers and/or the returns from investments in the financial market.
Collective risk models, introduced by Lundberg and Cram\'er, describe the evolution of the surplus of an insurance business with constant premiums rate, for the simplicity of arguments. This model, a compound Poisson process with drift, is referred to in the actuarial mathematics literature as the Cram\'er-Lundberg model. However, in practical situations, risk models with surplus-dependent premiums capture better the dynamics of the surplus of an insurance company. \citet{dividend} advised for a lower premium for higher surplus level to improve competitiveness, whereas a higher premium is needed for lower surplus level to reduce the probability of ruin.  \\
\\
Among surplus-dependent premiums, risk models with {\it risky} investments have been widely analyzed (see e.g. \cite{paulsen1993,investment, frolova, act}).  See \cite{paulsen1,paulsen2} for surveys on the topic. The special case of risk models with {\it linearly} dependent premiums can be interpreted as models with {\it riskless} investments, since the volatility of return on investments or the proportion of the capital invested in the risky asset is zero. Under this scenario, exact expressions of the ruin probability are derived for compound Poisson risk models with interest on surplus and exponential-type upper bounds for renewal risk models with interest (see \cite{caidickson2002,caidickson2003}). \citet{chelansurplus} investigate risk models with surplus-dependent premiums with dividend strategies and interest earning as a special case.\\
\\

Throughout this paper, we build on the method developed in \citet{exactasymptotic} to  extend the derivation of ruin probabilities to surplus-dependent premiums risk models with {\it Erlang} distributions (claim sizes or interarrival times). Recall from  \citet{exactasymptotic}, the risk model with surplus-dependent premiums is described by
\begin{equation} \label{model}
U(t)=u+\int_{0}^{t}p(U(s))ds-\sum_{k=1}^{N(t)}X_k,
\end{equation}
where $U(t)$ denotes the surplus at time $t$, and  $p(\cdot)$ is the premium rate at time $t$, a positive function of the current surplus $U(t)$. When $p(.)$ is constant, this model reduces to the classical collective risk model, see \cite{aaruin}. As in classical collective risk theory, ruin defines the first time the surplus becomes negative. For $T_u$, the time of ruin, given by
\begin{equation*}
T_u=\inf\{t\geqslant 0|U(t)<0\},
\end{equation*}
the probability of ruin with initial value $u$ is defined as
\begin{equation*}
\psi(u)=\mathbb{P}\left\{T_u<\infty|U(0)=u\right\}.
\end{equation*}

We focus on calculating ruin probabilities under Erlang claims and arrivals. Previously, \citet{willmotinterclaim} considered mixed Erlang claim size class when examining various properties associated with renewal risk processes with constant premium rates. Furthermore, \citet{willmotmixture} applied Erlang mixture to the claim size distribution when discussing the application of ruin-theoretic quantities. Various studies of ruin probabilities focus on risk model with interclaim times being Erlang($n$) distributed (see \cite{ligarridoerlangn,gerbershiumodel,lidicksonerlangn}) and Erlang(2) distributed (see \cite{dicksonhipperlang2,tsaisunerlang2,finiteerlang2}). \\

We use an algebraic approach to derive the equations satisfied by  the ruin probabilities, similar to the one from \citet{algebraic}, and further perform an asymptotic analysis of their solutions. We even solve them explicitely in a few instances. For perspective,  \citet{algebraic} introduced an algebraic approach to study the Gerber-Shiu function, and derived a linear ordinary differential equation (ODE) with constant coefficients for claims distribution with rational Laplace transform. Later in 2013, they extended this approach to an ODE with variable coefficients for surplus-dependent premiums risk models. Using method based on boundary value problems and Green's operators, they derived the explicit form of the ruin probability in the classical model with exponential claim sizes. \citet{exactasymptotic} extended the method to surplus-dependent premium models with exponential arrivals, for which they derived exact and asymptotic results for a few premium functions, when the claims were exponentially distributed. Here we extend to renewal models and Erlang claims.\\

The novelty of the paper consists on the explicit {\it asymptotic analysis} performed for reserve dependent premium with
Erlang distributed generic claim sizes or Erlang distributed generic interarrival times.
We separate the analysis between $p(\infty)=c$ and $p(\infty)=\infty$ and use the approximation theory of solutions of linear ordinary differential equations developed in \citet{Fed1993} to conclude the asymptotics of the ruin probabilities when initial reserves tend to infinity.\\

Among the premium functions exploding at infinity, i.e. $p(\infty)=\infty$, we consider the linear premium $p(u)=c+\varepsilon u$, in which $\varepsilon$ can be interpreted as the interest rate on the available surplus. {\it linear} premiums can be interpreted as investment of the company in bonds or risk-free assets.
When considering premiums that are {\it linearly} dependent on reserves,  we firstly derive {\it explicit formulas} for the ruin probabilities, using confluent geometric functions and their corresponding ODEs.
From these exact expressions we can easily determine their
asymptotics, only to match the ones obtained for general premiums dependent on reserves.\\

We show that when the investments are made on risk-free assets only, as bonds or treasury bills, the solvency is improved. We will look at the improvements on solvency when such investments are made, by analyzing the insurance risk models with or without investment returns, for claims and claim arrivals that are exponential or Erlang distributed. We compare them with the asymptotics of the equivalent ruin probabilities when the premium rate is fixed over time, to
measure the gain generated by this additional mechanism of binding the premium rates with the amount of reserve own by the insurance company.\\

In this paper we consider the three cases
\begin{itemize}
	\item[] (\lowercase\expandafter{\romannumeral1}) Exp$(\lambda)$ distributed interarrival times with Exp$(\mu)$ distributed claims sizes,
	\item[] (\lowercase\expandafter{\romannumeral2}) Erlang$(2,\lambda)$ distributed interarrival times with Exp$(\mu)$ distributed claims sizes,
	\item[] (\lowercase\expandafter{\romannumeral3}) Exp$(\lambda)$ distributed interarrival times with Erlang$(2,\mu)$ distributed claims sizes.
\end{itemize}
We consider two cases of premium functions:
\begin{enumerate}
\item[P1.] the premium function behaves like a constant at infinity
\begin{equation}\label{pierwszyprzypadek}
p(\infty)=c,\qquad p^\prime(u)= O\left(\frac{1}{u^2}\right);
\end{equation}
for $c>0$ or

\item[P2.] the premium function explodes at infinity, $p(\infty)=\infty$ as
\begin{equation}\label{drugiprzypadek}
p(u)=c+\sum_{i=1}^l\epsilon_i u^i,\qquad \epsilon_i, c>0.
\end{equation}
\end{enumerate}
The first case is satisfied by the rational and exponential premium functions.
The second case is satisfied by the linear and quadratic premium functions.\\

The paper is organized as folows. In Section 2, we introduce the Gerber-Shiu function and present the derivation of the boundary value problem for them in models with premium dependent on reserves and times and claims from distributions with rational Laplace transforms. We recall the results for ruin probabilities, in models with premiums dependent on reserves, general and linear premiums, when both inter-arrivals and claim sizes are exponentially distributed. In Sections 3 and 4, we perform the asymptotic analysis for the ruin probabilities for exponential and Erlang(2) distributed claim sizes and interarrival times, alternatively, for models with premiums dependent on reserves. In each section, for linear premiums, the exact ruin probabilities are derived and the asymptotics confirmed with those obtained for {\it general} premiums.  Section 5 is dedicated to comparing the asymtotic results, highlighting the gain generated, as in higher solvency, when dynamically adjusting the premium rates to surplus. Conclusions are given in Section 6.

\section{Ruin probabilities - method}
\label{method}

Ruin probability is sometimes seen as a particular case of the Gerber-Shiu function $\Phi(u)$ defined in \citet{gerbershiutime}. $\Phi(u)$ is given by
\begin{equation}
\label{gs}
\Phi(u)=\mathbb{E}[e^{-\delta T_u}\omega(U(T_u^-),|U(T_u))\mathbf{1}_{T_u<\infty}|U(0)=u].
\end{equation}
where $e^{-\delta T_{u}}$ is the discount factor, $\omega$ is the penalty function of surplus before ruin $U(T_{u}^{-})$ and deficit at ruin $U(T_{u})$. Thus, the ruin probability $\psi(u)$ is a special case of Gerber-Shiu function when $\delta =0$ and $\omega=1$.\\

Assuming that the distribution of the interclaim times $(\tau_k)_{k \geqslant 0}$ and the claim sizes $(X_k)_{k\geqslant0}$ have rational Laplace transform, the density functions $f_\tau(t)$, $f_X(x)$ satisfy linear ordinary differential equation
\begin{equation*}
\mathcal{L}_\tau\left(\frac{d}{dt}\right)f_\tau(t)=0,\quad \mathcal{L}_X\left(\frac{d}{dy}\right)f_X(y)=0,
\end{equation*}
with initial conditions
\begin{align*}
&{f_\tau}^{(k)}(0)=0 \quad(k=0,1,\ldots,n-2),\quad{f_\tau}^{(n-1)}(0)=\alpha_0,\\
&{f_X}^{(k)}(0)=0 \quad(k=0,1,\ldots,m-2),\quad{f_X}^{(m-1)}(0)=\beta_0,
\end{align*}
where
\begin{align*}
&\mathcal{L}_\tau\left(\frac{d}{dt}\right)=\left(\frac{d}{dt}\right)^n+\alpha_{n-1}\left(\frac{d}{dt}\right)^{n-1}+\cdots+\alpha_0,\\
&\mathcal{L}_X\left(\frac{d}{dx}\right)=\left(\frac{d}{dx}\right)^n+\beta_{n-1}\left(\frac{d}{dx}\right)^{n-1}+\cdots+\beta_0.
\end{align*}

For the risk models with surplus-dependent premiums, \citet{exactasymptotic} derived a compact integro-differential equation for $\Phi(u)$
\begin{equation}
\label{idesd}
\mathcal{L}_\tau\left(\delta-p(u)\frac{d}{du}\right)\Phi(u)=\alpha_0\left(\int_{0}^{u}\Phi(u-y)dF_X(y)+\omega(u)\right),
\end{equation}
where $\omega(x)=\int_{x}^{\infty}\omega(x,y-x)dF_X(y)$.\\
\\

For a Gerber-Shiu function, the coefficients of ODE are variables (non-constant), and the boundary value problem developed by \citep{exactasymptotic} is
\begin{equation}
\label{lodesd}
\mathcal{L}_X\left(\frac{d}{du}\right)\mathcal{L}_\tau\left(\delta-p(u)\frac{d}{du}\right)\Phi(u)=\alpha_0\beta_0\Phi(u)+\alpha_0\mathcal{L}_X
\left(\frac{d}{du}\right)\omega(u),
\end{equation}
exhibiting one regularity condition
\begin{equation*}
\Phi(\infty)=0
\end{equation*}
and $m$ initial conditions
\begin{equation*}
\Phi^{(k)}(0)=0\quad(k=0,\ldots,m-1).
\end{equation*}
\\
The general solution of this boundary value problem has the form
\begin{equation*}
\Phi(u)=\gamma_1s_1(u)+\cdots+\gamma_ms_m(u)+Gg(u),
\end{equation*}
where $s_i(u)$, $i=1,\ldots,m$ are $m$ stable solutions ($s_i(u)\rightarrow0$ as $u\rightarrow\infty$), $\gamma_i$ are constants determined by initial conditions, $g(u)=\alpha_0\mathcal{L}_X(\frac{d}{du})\omega(u)$, and $Gg(u)$ is the Green's operator for (\ref{lodesd}) (see \cite{exactasymptotic}).\\
\\
Again, the probability of ruin $\psi(u)$ is a special case of $\Phi(u)$ for $\delta=0$ and $\omega=1$. Thus one has
\begin{equation*}
\psi(u)=\gamma_1s_1(u)+\cdots+\gamma_ms_m(u).
\end{equation*}

In next sections we developed above theory to analyse the case when
either generic interarrival time or generic claim size has Gamma (Erlang) distribution. We start
for more easy case when both, generic claim and generic interarrival time, have exponential distributions.\\

For a classical compound Poisson process with exponential claims, the following explicit and asymptotic results for ruin probability $\psi(u)$ can be found in \cite{aaruin,exactasymptotic}.\\

{\bfseries{General premium.}}
For a classical compound Poisson process with exponential claims, the ruin probability $\psi(u)$ has the following explicit expression
\begin{equation}
\label{ar}
\psi(u)=\frac{\lambda\int_{u}^{\infty}e^{-\mu v+\int_{0}^{v}\frac{\lambda}{p(y)}dy}\frac{1}{p(v)}dv}{1+\lambda\int_{0}^{\infty}e^{-\mu v+\int_{0}^{v}\frac{\lambda}{p(y)}dy}\frac{1}{p(v)}dv}.
\end{equation}
\\
The asymptotic estimate of ruin probability for $p(\infty)=c$ is
\begin{equation*}
\psi(u)\sim\frac{\mu}{\lambda}Ce^{-\mu u+\lambda\int_0^u\frac{dw}{p(w)}},\quad u\rightarrow\infty
\end{equation*}
and for $p(\infty)=\infty$ is
\begin{equation*}
\psi(u)\sim\frac{\mu}{\lambda}C\frac{1}{p(u)} e^{-\mu u+\lambda\int_0^u\frac{dw}{p(w)}},\quad u\rightarrow\infty
\end{equation*}
where $C$ is a constant. We write $f(u)\sim g(u)$ for some functions $f$ and $g$ when $\lim_{u\to+\infty} f(u)/g(u)=1$.\\

{\bfseries{Linear premium.}}
The explicit form of ruin probability $\psi(u)$ is
\begin{equation}
\label{l1}
\psi_{l,1}(u)=\frac{\lambda\varepsilon^{\lambda/\varepsilon-1}}{\mu^{\lambda/\varepsilon}c^{\lambda/\varepsilon}e^{-\mu c/\varepsilon}+\lambda\varepsilon^{\lambda/\varepsilon-1}\Gamma(\frac{\mu c}{\varepsilon},\frac{\lambda}{\varepsilon})}\Gamma(\frac{\mu(c+\varepsilon u)}{\varepsilon},\frac{\lambda}{\varepsilon}),
\end{equation}
where $\Gamma(x,\eta)$ is the incomplete gamma function defined as
\begin{equation*}
\Gamma(x,\eta)=\int_{x}^{\infty}t^{\eta-1}e^{-t}dt.
\end{equation*}

Moreover, when $p(u)=c+\varepsilon u$, we have
\begin{equation}
\label{asymptotic1}
\psi(u)\sim\frac{\mu}{\lambda c^{\lambda/\varepsilon}}Ce^{-\mu u}(c+\varepsilon u)^{\frac{\lambda}{\varepsilon}-1},\quad \text{as }u\rightarrow\infty.
\end{equation}

\section{Erlang$(2,\lambda)$ distributed interarrival times with Exp$(\mu)$ distributed claims sizes,}
\label{erlangexp}
Let the claim sizes $(X_k)_{k\geqslant 0}$ be exponentially distributed and interarrival times $(\tau_k)_{k\geqslant 0}$ be Erlang$(2,\lambda)$ distributed, that is their density functions are
\begin{equation*}
f_X(x)=e^{-\mu x},\ x\geqslant0 \quad\text{and}\quad f_\tau(t)=\lambda^2te^{-\lambda t},\ t\geqslant0.
\end{equation*}
We denote by  $\psi_{l,2}(u)$ and $\Phi_{l,2}(u)$  the ruin probability and Gerber-Shiu function in this case.

\subsection{General premium}
Based on the technique as in \citet{algebraic,exactasymptotic}, the boundary value problem (\ref{lodesd}) becomes
\begin{equation*}
\left[\left(\frac{d}{du}+\mu\right)\left(\delta-p(u)\frac{d}{du}+\lambda\right)^2-\lambda^2\mu\right]\Phi_{l,2}(u)=Gg(u),\quad u\geqslant0.
\end{equation*}
\\
For the special case $\delta=0$ and $\omega=1$, $g(u)=0$, the ODE of the ruin probability $\psi_{l,2}(u)$ has the form
\begin{equation}
\label{er}
\left[\left(\frac{d}{du}+\mu\right)\left(-p(u)\frac{d}{du}+\lambda\right)^2-\lambda^2\mu\right]\psi(u)=0,\quad u\geqslant0
\end{equation}
and
\begin{equation*}
\psi_{l,2}(u)=\gamma_{21}s_{21}(u),
\end{equation*}
where $s_{21}(u)$ is a stable solution and $\gamma_{21}$ is a constant  to be determined by the initial conditions.\\
\\
Expanding ODE (\ref{er}) leads to
\begin{align*}
p^2(u)\psi_{l,2}'''(u)+(2p'(u)p(u)-2\lambda p(u)+\mu p^2(u))\psi_{l,2}''(u)+(\lambda^2-2\lambda p'(u)-2\lambda\mu p(u))\psi_{l,2}'(u)=0.
\end{align*}\\
This is a third-order ODE with variable coefficients. 
Considering the third-order as second-order ODE in $h_{l,2}(u)=\psi_{l,2}'(u)$, one has
\begin{align}
\label{e2}
p^2(u)h_{l,2}''(u)+(2p'(u)p(u)-2\lambda p(u)+\mu p^2(u))h_{l,2}'(u)+(\lambda^2-2\lambda p'(u)-2\lambda\mu p(u))h_{l,2}(u)=0.
\end{align}\\

In order to perform the asymptotic analysis as in \citep[p. 250]{Fed1993},
we consider the characteristic equation of \eqref{e2} when $p(u)=c$.
Let
$\hat{\rho}_1$ and $\hat{\rho}_2$ be solutions of the square equation
$$\rho^2+\frac{-2\lambda c+\mu c^2}{c^2}\rho+\frac{\lambda^2-2\lambda\mu c}{c^2}=0,$$
that is,
\begin{equation}\label{rhos}
\hat{\rho}_{i}=\frac{2\lambda c-\mu c^2\pm\sqrt{(2\lambda c-\mu c^2)^2+4\lambda c^2(2\mu c-\lambda)}}{2c^2}.
\end{equation}
Moreover, let
\begin{equation}\label{rho1}
\rho_1(u)
=\frac{1}{2}\left(-q_1(u)-\sqrt{q_1^2(u)-4q_0(u)}\right)
\end{equation}
and
\begin{equation}\label{rho1}
\rho_2(u)
=\frac{1}{2}\left(-q_1(u)+\sqrt{q_1^2(u)-4q_0(u)}\right)
\end{equation}
be solutions of the characteristic equation
$\rho^2+q_1(u)\rho+q_0(u)=0$, where
$$q_1(u) =\frac{2p'(u)p(u)-2\lambda p(u)+\mu p^2(u)}{p^2(u)}$$
and
$$q_0(u)=\frac{\lambda^2-2\lambda p'(u)-2\lambda\mu p(u)}{p^2(u)}.$$
Further, as in \citep{Fed1993}, denote
\begin{equation}\label{rhoone}
\rho_i^{(1)}(u)=\frac{-\rho_i'(u)}{2 \rho_i(u)+q_1(u)},\qquad i=1,2.\end{equation}

\begin{theorem}\label{thm:1}
Let $C_i$ ($i=1,2,3$) be some constants. If  \eqref{pierwszyprzypadek} holds
with
\begin{align}\label{safeload1}
\frac{2c}{\lambda}>\frac{1}{\mu},
\end{align}
then
\begin{equation}\label{firstas2}
\psi_{l,2}(u)\sim -\frac{C_1}{\hat{\rho}_1}e^{\hat{\rho}_1 u}
\end{equation}
where $\hat{\rho}_1<0$.
%
If (\ref{drugiprzypadek}) holds then
\begin{equation}\label{secondtasb}
\psi_{l,2}(u)\sim C_3\int_u^\infty \exp\left\{\int_0^y (\rho_1(z)+ \rho_1^{(1)}(z)dz\right\}dy.
\end{equation}
\end{theorem}
\begin{remark}\rm
Under complimentary to \eqref{safeload1} assumption
\begin{align*}\label{safeload2}
\frac{2c}{\lambda}<\frac{1}{\mu}
\end{align*}
we have $\hat{\rho}_{1,2}>0$ and hence both asymptotic special solutions are unstable.
Their difference might still tend to zero but
\citet{Fed1993} theory is not sufficient precise to recover
the finer asymptotics in this case.
\end{remark}
\begin{proof}
Note that in the case of premium function  \eqref{pierwszyprzypadek} we have
$$q_1(u)= \frac{-2\lambda c+\mu c^2}{c^2} + O\left(\frac{1}{u^2}\right)$$
and
$$q_0(u)=\frac{\lambda^2-2\lambda\mu c}{c^2}+ O\left(\frac{1}{u^2}\right).$$
Further,  under assumption \eqref{safeload1}, $\hat{\rho}_1<0$ and $\hat{\rho}_2>0$ for $\hat{\rho}_{1,2}$ defined in \eqref{rhos}.
Then Conditions 1) and 2) of \citet[p. 250]{Fed1993} are satisfied
and hence choosing the stable solution (tending to zero as $u$ tends to infinity) we have
\begin{equation}\label{firstas}
h_{l,2}(u)\sim e^{\hat{\rho}_1 u}
\end{equation}
if \eqref{safeload1} is satisfied.
Thus asymptotics \eqref{firstas2} holds true.

In the second case of premium function (\ref{drugiprzypadek}) observe that the
solutions of the characteristic equation
$\rho^2+q_1(u)\rho+q_0(u)=0$ satisfy:
\begin{equation}\label{rho1}
\rho_1(u)
=\frac{1}{2}\left(-q_1(u)-\sqrt{q_1^2(u)-4q_0(u)}\right)\sim -\mu -\frac{q_0(u)}{\mu}\sim -\mu-\frac{2\lambda}{\epsilon_l} u^{-l}
\end{equation}
and
\begin{equation}\label{rho1b}
\rho_2(u)
=\frac{1}{2}\left(-q_1(u)+\sqrt{q_1^2(u)-4q_0(u)}\right)\sim -\frac{2\lambda}{\epsilon_l} u^{-l}.
\end{equation}
Moreover, in this case $q_0(u)\sim \frac{-2\lambda \mu}{\epsilon_l}u^{-l}$ and
$q_1(u)\sim \mu$.
Thus Conditions 1), 2') and (19) of \citet[p. 254]{Fed1993} are satisfied
and
we can conclude \eqref{secondtasb}. Note that from \eqref{rho1b}
$$\int_u^\infty \exp\left\{\int_0^y (\rho_2(z)+ \rho_2^{(1)}(z)dz\right\}dy$$
tends to infinity. Hence by \eqref{rho1}
only
$$\int_u^\infty \exp\left\{\int_0^y (\rho_1(z)+ \rho_1^{(1)}(z)dz\right\}dy$$
can produce the stable asymptotics \eqref{secondtasb}
in a sense that it tends to zero as $u$ tends to infinity.
\end{proof}

Observe that indeed in all considered cases $\psi_{l,2}(u)\rightarrow 0$ as $u\rightarrow+\infty$, that is, we choose
the asymptotics of stable solutions.

\subsection{Linear premium}
Now we perform the asymptotic analysis of the special case of linear premium rate which corresponds to investments of reserves into bonds
with interest rate $ \varepsilon>0$.
Substituting $p(u)=c+\varepsilon u$ into ODE (\ref{e2}), we have
\begin{align}
\label{e2l}
\nonumber&(c+\varepsilon u)^2h_{l,2}''(u)+(2\varepsilon(c+\varepsilon u)-2\lambda(c+\varepsilon u)+\mu (c+\varepsilon u)^2)h_{l,2}'(u)\\
&+(\lambda^2-2\lambda\varepsilon-2\lambda\mu(c+\varepsilon u))h_{l,2}(u)=0.
\end{align}

Before we solve this equation and perform the asymptotic analysis we will
show how the asymptotics of $\psi_{l,2}$ can be derived from Theorem \ref{thm:1}.
In this case, we have
\begin{equation*}
q_1(u) =\frac{2\varepsilon-2\lambda}{c+\varepsilon u}+\mu \ {\rm and}\ q_0(u)=\frac{\lambda^2-2\lambda\varepsilon}{(c+\varepsilon u)^2}-\frac{2\lambda\mu}{c+\varepsilon u}.
\end{equation*}
Further, the discriminant is
$$q_1^2(u)-4q_0(u)=\frac{4\varepsilon^2}{(c+\varepsilon u)^2}+\frac{4\mu\varepsilon+4\lambda\mu}{c+\varepsilon u}+\mu^2$$
and therefore
\begin{align*}
\rho_1(u)&=\frac{1}{2}\left(-q_1(u)-\sqrt{q_1^2(u)-4q_0(u)}\right)\\
&=-\frac{1}{2}\mu-\frac{\varepsilon-\lambda}{c+\varepsilon u}-\frac{1}{2}\sqrt{\mu^2+\frac{4\mu(\varepsilon+\lambda)}{c+\varepsilon u}+\frac{4\varepsilon^2}{(c+\varepsilon u)^2}}.
\end{align*}
Applying Taylor expansion, we can conclude that
\begin{equation*}
\sqrt{q_1^2(u)-4q_0(u)}\sim\mu+\frac{2(\varepsilon+\lambda)}{c+\varepsilon u},\quad\text{as } u\rightarrow\infty.
\end{equation*}
Additionally, observe that
\begin{equation*}
q'_1(u)=-\frac{\varepsilon(2\varepsilon-2\lambda)}{(c+\varepsilon u)^2}\;{\rm and}\; \left(\sqrt{q_1^2(u)-4q_0(u)}\right)'=\frac{-\frac{4\varepsilon^3}{(c+\varepsilon u)^3}-\frac{\varepsilon(2\varepsilon\mu+2\lambda\mu)}{(c+\varepsilon u)^2}}{\sqrt{\frac{4\varepsilon^2}{(c+\varepsilon u)^2}+\frac{4\mu\varepsilon+4\lambda\mu}{c+\varepsilon u}+\mu^2}}.
\end{equation*}
This gives that
\begin{align*}
\frac{\rho'_1(u)}{\sqrt{q_1^2(u)-4q_0(u)}}&=\frac{\frac{1}{2}\left(-q_1(u)-\sqrt{q_1^2(u)-4q_0(u)}\right)'}{\sqrt{\frac{4\varepsilon^2}{(c+\varepsilon u)^2}+\frac{4\mu\varepsilon+4\lambda\mu}{c+\varepsilon u}+\mu^2}}\\
&=\frac{\frac{\varepsilon(\varepsilon-\lambda)}{(c+\varepsilon u)^2}}{\sqrt{\frac{4\varepsilon^2}{(c+\varepsilon u)^2}+\frac{4\mu\varepsilon+4\lambda\mu}{c+\varepsilon u}+\mu^2}}+\frac{\frac{2\varepsilon^3}{(c+\varepsilon u)^3}+\frac{\varepsilon(\varepsilon\mu+\lambda\mu)}{(c+\varepsilon u)^2}}{\frac{4\varepsilon^2}{(c+\varepsilon u)^2}+\frac{4\mu\varepsilon+4\lambda\mu}{c+\varepsilon u}+\mu^2}.
\end{align*}
Using \eqref{rhoone} we finally derive
\begin{align*}
\rho_1(u)+ \rho_1^{(1)}(u)=\rho_1(u)+\frac{\rho'_1(u)}{\sqrt{q_1^2(u)-4q_0(u)}}\sim-\mu-\frac{2\varepsilon}{c+\varepsilon u},\quad\text{as } u\rightarrow\infty.
\end{align*}
Thus, for $u\rightarrow\infty$,
\begin{align*}
\exp\left\{\int_0^y (\rho_1(z)+ \rho_1^{(1)}(z)dz\right\}=e^{-\mu y}\left(\frac{c+\varepsilon y}{c}\right)^{-2}
\end{align*}
and by \eqref{secondtasb}
\begin{equation}\label{aslinearjeden}
\psi_{l,2}(u)\sim C_3\int_u^\infty e^{-\mu y}\left(\frac{c+\varepsilon y}{c}\right)^{-2} dy
\end{equation}
for some constant $C_3$.
The same asymptotics can be derived by solving \eqref{e2l} explicitly.
Note that \eqref{e2l} is the general confluent equation 13.1.35 in \citet[p. 505]{ashandbook}, which has the form
\begin{align*}
w''(z)&+\left[\frac{2A}{Z}+2f'(z)+\frac{bh'(z)}{h(z)}-h'(z)-\frac{h''(z)}{h'(z)}\right]w'(z)\\
+&\left[\left(\frac{bh'(z)}{h(z)}-h'(z)-\frac{h''(z)}{h'(z)}\right)\left(\frac{A}{Z}+f'(z)\right)+\frac{A(A-1)}{Z^2}+\frac{2Af'(z)}{Z}+f''(z)+f'^2(z)-\frac{ah'^2(z)}{h(z)}\right]w(z)=0.
\end{align*}

For our ODE \eqref{e2l}, let
\begin{align*}
&Z=\frac{c+\varepsilon u}{\varepsilon},\;f(Z)=h(Z)=\mu Z,\\
&A=\frac{1}{2}-\frac{\lambda}{\varepsilon}-\frac{1}{2}\sqrt{1+\frac{4\lambda}{\varepsilon}},\;a=1+\frac{\varepsilon+2\lambda+\varepsilon\sqrt{1+\frac{4\lambda}{\varepsilon}}}{2\varepsilon},\;b=1+\sqrt{1+\frac{4\lambda}{\varepsilon}},
\end{align*}
the corresponding solutions are
\begin{align*}
h_{l,21}(u)=&C_{21}e^{-\mu u}(c+\varepsilon u)^{-\frac{1}{2}+\frac{\lambda}{\varepsilon}+\frac{1}{2}\sqrt{1+\frac{4\lambda}{\varepsilon}}}\\
&\cdot {\rm M}\left(1+\frac{\varepsilon+2\lambda+\varepsilon\sqrt{1+\frac{4\lambda}{\varepsilon}}}{2\varepsilon},1+\sqrt{1+\frac{4\lambda}{\varepsilon}},\frac{\mu(c+\varepsilon u)}{\varepsilon}\right),
\end{align*}
and
\begin{align*}
h_{l,22}(u)=&C_{22}e^{-\mu u}(c+\varepsilon u)^{-\frac{1}{2}+\frac{\lambda}{\varepsilon}+\frac{1}{2}\sqrt{1+\frac{4\lambda}{\varepsilon}}}\\
&\cdot {\rm U}\left(1+\frac{\varepsilon+2\lambda+\varepsilon\sqrt{1+\frac{4\lambda}{\varepsilon}}}{2\varepsilon},1+\sqrt{1+\frac{4\lambda}{\varepsilon}},\frac{\mu(c+\varepsilon u)}{\varepsilon}\right),
\end{align*}
where $C_{21}$ and $C_{22}$ are constants.\\

From p. 504 of \cite{ashandbook} we know that
\begin{equation*}
{\rm M}(a,b,z)\sim \frac{\Gamma(b)}{\Gamma(a)} e^z z^{a-b}\text{ and } {\rm U}(a,b,z)\sim z^{-a}, \text{ as } z\rightarrow\infty.
\end{equation*}
Thus, we have $h_{l,21}(u)\rightarrow\infty$ and $h_{l,22}(u)\rightarrow0$ for $u\rightarrow\infty$.\\

Recalling that $s_{2}(u)$ is the stable solution, we have
\begin{align*}
s_{2}(u)=C\int_{u}^{\infty}&e^{-\mu v}(c+\varepsilon v)^{-\frac{1}{2}+\frac{\lambda}{\varepsilon}+\frac{1}{2}\sqrt{1+\frac{4\lambda}{\varepsilon}}}\\&
\cdot {\rm U}\left(1+\frac{\varepsilon+2\lambda+\varepsilon\sqrt{1+\frac{4\lambda}{\varepsilon}}}{2\varepsilon},1+\sqrt{1+\frac{4\lambda}{\varepsilon}},\frac{\mu(c+\varepsilon v)}{\varepsilon}\right)dv.
\end{align*}\\
Thus, as $u\rightarrow\infty$,
the ruin probability has the following asymptotics:
\begin{equation*}
\psi_{l,2}(u)\sim C\int_{u}^{\infty}e^{-\mu y}\cdot(c+\varepsilon y)^{-\frac{1}{2}+\frac{\lambda}{\varepsilon}+\frac{1}{2}\sqrt{1+\frac{4\lambda}{\varepsilon}}}\cdot{\left(\frac{\mu(c+\varepsilon y)}{\varepsilon}\right)}^{-1-\frac{\varepsilon+2\lambda+\varepsilon\sqrt{1+\frac{4\lambda}{\varepsilon}}}{2\varepsilon}}dy
\end{equation*}
equivalent to
\begin{equation*}
\psi_{l,2}(u)\sim C {\Big(\frac{\mu}{\varepsilon}\Big)}^{-1-\frac{\varepsilon+2\lambda+\varepsilon\sqrt{1+\frac{4\lambda}{\varepsilon}}}{2\varepsilon}}\int_{u}^{\infty}e^{-\mu y}\cdot(c+\varepsilon y)^{-2}dy,\quad \mbox{as $u\rightarrow\infty$.}
\end{equation*}
By \eqref{aslinearjeden} this asymptotic behaviour is the same as the one derived using the Theorem \ref{thm:1}.
Furthermore, one can simplify the above asymptotics by applying the integration-by-parts formula into
\begin{align}
\label{l2}
\nonumber \psi_{l,2}(u)&\sim C\cdot{\Big(\frac{\mu}{\varepsilon}\Big)}^{-1-\frac{\varepsilon+2\lambda+\varepsilon\sqrt{1+\frac{4\lambda}{\varepsilon}}}{2\varepsilon}}
\cdot\frac{1}{\varepsilon}\left(\frac{e^{-\mu u}}{c+\varepsilon u}-\mu\int_{u}^{\infty}\frac{e^{-\mu v}}{c+\varepsilon v}dv\right)\\
&=C\cdot{\left(\frac{\mu}{\varepsilon}\right)}^{-1-\frac{\varepsilon+2\lambda+\varepsilon\sqrt{1+\frac{4\lambda}{\varepsilon}}}{2\varepsilon}}
\cdot\frac{1}{\varepsilon}\left(\frac{e^{-\mu u}}{c+\varepsilon u}-\frac{\mu}{\varepsilon}e^{\frac{\mu}{\varepsilon}c}\Gamma\left(\frac{\mu}{\varepsilon}(c+\varepsilon u),0\right)\right),\quad \mbox{as $u\rightarrow\infty$.}
\end{align}\\

\section{Exp$(\lambda)$ distributed interarrival times with Erlang$(2,\mu)$ distributed claims sizes}
\label{experlang}
Let the claim sizes $(X_k)_{k\geqslant 0}$ be Erlang$(2,\mu)$ distributed and the interarrival times $(\tau_k)_{k\geqslant 0}$ be Exp$(\lambda)$ distributed, that is,
\begin{equation*}
f_X(x)=\mu^2xe^{-\mu x},\ x\geqslant0 \quad\text{and}\quad f_\tau(t)=\lambda e^{-\lambda t},\ t\geqslant0.
\end{equation*}
We denote by  $\psi_{l,3}(u)$ and $\Phi_{l,3}(u)$  the ruin probability and Gerber-Shiu function in this case.

\subsection{General premium}
Applying the same technique as in \citet{algebraic,exactasymptotic}, the boundary value problem (\ref{lodesd}) becomes
\begin{equation*}
\left[\left(\frac{d}{du}+\mu\right)^2\left(\delta-p(u)\frac{d}{du}+\lambda\right)-\lambda\mu^2\right]\Phi(u)_{l,3}=Gg(u),\quad u\geqslant0.
\end{equation*}
\\
For $\delta=0$ and $\omega=1$, $g(u)=0$ and the ODE of $\psi_{l,3}(u)$ has the following form
\begin{equation}
\label{er2}
\left[\left(\frac{d}{du}+\mu\right)^2\left(-p(u)\frac{d}{du}+\lambda\right)-\lambda\mu^2\right]\psi_{l,3}(u)=0,\quad u\geqslant0,
\end{equation}
equivalent to
\begin{equation}
	\left(\frac{d^2}{du^2}+2\mu\frac{d}{du}+\mu^2\right)(-p(u)\psi_{l,3}'(u)+\lambda\psi_{l,3}(u))=\lambda\mu^2\psi_{l,3}(u).
\end{equation}
Hence, one can rewrite it as
\begin{equation*}
p(u)\psi_{l,3}'''(u)+(2p'(u)+2\mu p(u)-\lambda)\psi_{l,3}''(u)+(p''(u)+2\mu p'(u)+\mu^2p(u)-2\mu\lambda)\psi_{l,3}'(u)=0.
\end{equation*}\\

Denoting $h_{l,3}(u)=\psi_{l,3}'(u)$, we have the following equation in $h_{l,3}(u)$
\begin{equation}\label{Main2}
p(u)h_{l,3}''(u)+(2p'(u)+2\mu p(u)-\lambda)h_{l,3}'(u)+(p''(u)+2\mu p'(u)+\mu^2p(u)-2\mu\lambda)h_{l,3}(u)=0.
\end{equation}

We first analyze general premium rate.
We denote now, as in \citep{Fed1993}
$$\tilde{q}_1(u) =\frac{2p'(u)+2\mu p(u)-\lambda}{p(u)}$$
and
$$\tilde{q}_0(u)=\frac{p''(u)+2\mu p'(u)+\mu^2p(u)-2\mu\lambda}{p(u)}.$$
Denote by
$\tilde{\rho}_{1,2}$
the roots of the quadratic equation
$$\rho^2+\frac{2\mu c-\lambda}{c}\rho+\frac{\mu^2c-2\mu\lambda}{c}=0,$$
that is, for $i=1,2$,
$$\tilde{\rho}_{i}=\frac{\lambda - 2\mu c\pm\sqrt{(\lambda - 2\mu c)^2+4\mu c(2\lambda-\mu c)}}{2c}.$$
Further, let
\[\tilde{\rho}_i^{(1)}(u)=\frac{-\tilde{\rho}_i'(u)}{2 \tilde{\rho}_i(u)+q_1(u)},\qquad i=1,2\]
with
\begin{equation}\label{rho1}
\tilde{\rho}_1(u)
=\frac{1}{2}\left(-\tilde{q}_1(u)-\sqrt{\tilde{q}_1^2(u)-4\tilde{q}_0(u)}\right)
\end{equation}
and
\begin{equation}\label{rho1}
\tilde{\rho}_2(u)
=\frac{1}{2}\left(-\tilde{q}_1(u)+\sqrt{\tilde{q}_1^2(u)-4\tilde{q}_0(u)}\right)
\end{equation}
being solutions of the characteristic equation
$\rho^2+q_1(u)\rho+q_0(u)=0$.

\begin{theorem}\label{thm:2}
Let $C_i$ ($i=1,2,3,4$) be some constants. If  \eqref{pierwszyprzypadek} holds
then
\begin{equation}\label{firstasc2}
\psi_{l,2}(u)\sim -\frac{C_1}{\tilde{\rho}_1}e^{\tilde{\rho}_1 u}
\end{equation}
where $\tilde{\rho}_{1}<0$ for
\begin{align}\label{safeloadc1}
\lambda>\frac{\mu c}{2}
\end{align}
and
\begin{equation}\label{firstas2cb}
\psi_{l,2}(u)\sim  -\frac{C_1}{\tilde{\rho}_1}e^{\tilde{\rho}_1 u}-\frac{C_2}{\tilde{\rho}_2}e^{\tilde{\rho}_2 u}
\end{equation}
where $\tilde{\rho}_{1,2}<0$ for
\begin{align}\label{safeloadc2}
\lambda<\frac{\mu c}{2}
\end{align}
Moreover, if (\ref{drugiprzypadek}) holds then
\begin{equation}\label{secondtasbc}
\psi_{l,2}(u)\sim C_3\int_u^\infty \exp\left\{\int_0^y (\tilde{\rho}_1(z)+ \tilde{\rho}_1^{(1)}(z)dz\right\}dy +
C_4\int_u^\infty \exp\left\{\int_0^y (\tilde{\rho}_2(z)+ \tilde{\rho}_2^{(1)}(z)dz\right\}dy.
\end{equation}
\end{theorem}
\begin{proof}
Note that in the first case of premium function \eqref{pierwszyprzypadek} we have
$$\tilde{q}_1(u)= \frac{2\mu c-\lambda}{c}  + O\left(\frac{1}{u^2}\right)$$
and
$$\tilde{q}_0(u)=\frac{\mu^2c-2\mu\lambda}{c}+ O\left(\frac{1}{u^2}\right).$$
Here
$\tilde{\rho}_1$
and $\tilde{\rho}_2$
are different. Then Conditions 1) and 2) of \citet[p. 250]{Fed1993} are satisfied
and hence choosing the stable solution (tending to zero as $u$ tends to infinity)
\begin{equation}\label{firstasc}
h_{l,2}(u)\sim e^{\tilde{\rho}_1 u}
\end{equation}
for $\tilde{\rho}_2>0$ or
\begin{equation}\label{firstascb}
h_{l,2}(u)\sim e^{\tilde{\rho}_1 u}+ e^{\tilde{\rho}_2 u}
\end{equation}
for $\tilde{\rho}_2<0$.
Similarly like in the proof of Theorem \ref{thm:1}, this observation completes the proof of
\eqref{firstasc2} and \eqref{firstas2cb}.\\

In the case of premium function (\ref{drugiprzypadek}) observe that
solutions of the characteristic equation
$\rho^2+q_1(u)\rho+q_0(u)=0$ converge for $u\rightarrow\infty$ to
\begin{equation}\label{rho1c}
\tilde{\rho}_1(u)
=\frac{1}{2}\left(-\tilde{q}_1(u)-\sqrt{\tilde{q}_1^2(u)-4\tilde{q}_0(u)}\right)\rightarrow -\mu
\end{equation}
and
\begin{equation}\label{rho1bc}
\tilde{\rho}_2(u)
=\frac{1}{2}\left(-\tilde{q}_1(u)+\sqrt{\tilde{q}_1^2(u)-4\tilde{q}_0(u)}\right) \rightarrow -\mu,
\end{equation}
since in this case $q_0(u)\rightarrow \mu^2$ and
$q_1(u)\rightarrow 2\mu$.
Although we are not in the set-up of asymptotically simple roots equation (9) in \citet[p. 251]{Fed1993}
\citet[(9)]{Fed1993} still holds true.
Observe now that for large $u$,
$$\tilde{\rho}_1(u)-\tilde{\rho}_2(u)+\tilde{\rho}_1^{(1)}(u)-\tilde{\rho}_2^{(1)}(u)$$
does not change sign.
Finally,
$$q_1'(u)\sim \frac{-2l}{u^2},\qquad q_1''(u)\sim \frac{4l}{u^3}$$
and
$$q_0'(u)\sim \frac{-2\mu l}{u^2} ,\qquad q_0''(u)\sim \frac{4\mu l}{u^3}.$$
Thus Conditions 1), 2) and (9) of \citet[p. 251-252]{Fed1993} are satisfied
and, similarly like in the proof of Theorem \ref{thm:1},
we can conclude the proof of \eqref{secondtasbc}.
\end{proof}

\subsection{Linear premium}
Using the same method as in the previous case and considering linear premium $p(u)=c+\varepsilon u$, one has
\begin{equation}
\label{e2cl}
(c+\varepsilon u)h_{l,3}''(u)+(2\varepsilon+2\mu (c+\varepsilon u)-\lambda)h_{l,3}'(u)+(2\mu\varepsilon+\mu^2(c+\varepsilon u)-2\mu\lambda)h_{l,3}(u)=0.
\end{equation}

As in the previous section,
before we solve this equation explicitly and then using this solution to perform the asymptotic analysis, we will first
show how the asymptotic behaviour of $\psi_{l,3}$ can be derived from Theorem \ref{thm:2}. Note that
in this case, we denote
\begin{equation*}
\tilde{q}_1(u) =\frac{2\varepsilon-\lambda}{c+\varepsilon u}+2\mu \ {\rm and}\ \tilde{q}_0(u)=\frac{2\mu\varepsilon-2\mu\lambda}{c+\varepsilon u}+\mu^2.
\end{equation*}
Further, we have the discriminant
$$\tilde{q}_1^2(u)-4\tilde{q}_0(u)=\frac{(2\varepsilon-\lambda)^2}{(c+\varepsilon u)^2}+\frac{4\lambda\mu}{c+\varepsilon u}$$
and hence for $i=1,2$,
\begin{align*}
\tilde{\rho}_{i}(u)&=\frac{1}{2}\left(-\tilde{q}_1(u)\pm\sqrt{\tilde{q}_1^2(u)-4\tilde{q}_0(u)}\right)\\
&=-\mu-\frac{2\varepsilon-\lambda}{2(c+\varepsilon u)}\pm\frac{1}{2}\sqrt{\frac{(2\varepsilon-\lambda)^2}{(c+\varepsilon u)^2}+\frac{4\lambda\mu}{c+\varepsilon u}}.
\end{align*}
Thus
\begin{align*}
\frac{\tilde{\rho}'_{i}(u)}{\sqrt{\tilde{q}_1^2(u)-4\tilde{q}_0(u)}}&=\frac{\frac{1}{2}\left(-\tilde{q}_1(u)\pm\sqrt{\tilde{q}_1^2(u)-4\tilde{q}_0(u)}\right)'}{\sqrt{\frac{(2\varepsilon-\lambda)^2}{(c+\varepsilon u)^2}+\frac{4\lambda\mu}{c+\varepsilon u}}}\\
&=\frac{\varepsilon(2\varepsilon-\lambda)}{2(c+\varepsilon u)\sqrt{(2\varepsilon-\lambda)^2+4\lambda\mu(c+\varepsilon u)}}\pm\frac{-\frac{\varepsilon(2\varepsilon-\lambda)}{c+\varepsilon u}-2\varepsilon\lambda\mu}{2(2\varepsilon-\lambda)^2+8\lambda\mu(c+\varepsilon u)}\\
&\sim\mp\frac{\varepsilon}{4(c+\varepsilon u)},\quad \mbox{as $u\rightarrow\infty$}.
\end{align*}
Also
\begin{equation*}
\tilde{\rho}_{i}(u)+ \tilde{\rho}_{i}^{(1)}(u)=\tilde{\rho}_{i}(u)\mp\frac{\tilde{\rho}'_{i}(u)}{\sqrt{\tilde{q}_1^2(u)-4\tilde{q}_0(u)}}\sim-\mu+
\left(-\frac{3\varepsilon}{4}+\frac{\lambda}{2}\right)\frac{1}{c+\varepsilon u}\pm\sqrt{\frac{\lambda\mu}{c+\varepsilon u}},\quad \mbox{as $u\rightarrow\infty$}.
\end{equation*}
We can then conclude that
\begin{equation*}
\exp\left\{\int_0^y (\tilde{\rho}_i(z)+ \tilde{\rho}_i^{(1)}(z)dz\right\}=e^{-\mu y\pm\frac{2}{\varepsilon}\left(\sqrt{\lambda\mu(c+\varepsilon y)}-\sqrt{\lambda\mu c}\right)}\left(\frac{c+\varepsilon y}{c}\right)^{-\frac{3}{4}+\frac{\lambda}{2\varepsilon}}
\end{equation*}
and by Theorem \ref{thm:2} we have that
\begin{align}
\psi_{l,3}(u)\sim&
C_3\int_u^\infty e^{-\mu y+\frac{2}{\varepsilon}\left(\sqrt{\lambda\mu(c+\varepsilon y)}-\sqrt{\lambda\mu c}\right)}\left(\frac{c+\varepsilon y}{c}\right)^{-\frac{3}{4}+\frac{\lambda}{2\varepsilon}}dy\nonumber\\
&+
C_4\int_u^\infty e^{-\mu y-\frac{2}{\varepsilon}\left(\sqrt{\lambda\mu(c+\varepsilon y)}-\sqrt{\lambda\mu c}\right)}\left(\frac{c+\varepsilon y}{c}\right)^{-\frac{3}{4}+\frac{\lambda}{2\varepsilon}}dy\label{asympdwa}
\end{align}
for some constants $C_3$ and $C_4$ as $u\rightarrow \infty$.\\

The same asymptotic behaviour can be observed by first
solving ODE \eqref{e2cl} explicitly. Note that for a general Bessel equation described by \citet[p. 211]{sherwood1939applied} is
\begin{equation*}
x^2\frac{d^2y}{dx^2}+\left[(1-2m)x-2\alpha x^2\right]\frac{dy}{dx}+\left[p^2a^2x^{2p}+\alpha^2x^2+\alpha(2m-1)x+m^2-p^2n^2\right]=0,
\end{equation*}
and the solution involves Bessel functions (see \citet[p. 211]{sherwood1939applied} and \citet[p. 460]{logan2012environmental} for details).\\

For our ODE \eqref{e2cl}, let
\begin{equation*}
x=c+\varepsilon u,\; m=-\frac{1}{2}+\frac{\lambda}{2\varepsilon},\; \alpha=-\frac{\mu}{\varepsilon},\; p=\frac{1}{2},\; p^2a^2=-\frac{\lambda\mu}{\varepsilon^2},\; n=-1+\frac{\lambda}{\varepsilon},
\end{equation*}
and employing the property ${\rm K}_{-v}(z)={\rm K}_{v}(z)$ (see \citet[p. 375]{ashandbook}), it can be verified that the corresponding solution is
\begin{align*}
h_{l,3}(u)=&C_{31}e^{-\frac{\mu}{\varepsilon}(c+\varepsilon v)}\cdot(c+\varepsilon v)^{-\frac{1}{2}+\frac{\lambda}{2\varepsilon}}\cdot {\rm BesselI}\Big[-1+\frac{\lambda}{\varepsilon},2\sqrt{\frac{(v+\frac{c}{\varepsilon})\lambda\mu}{\varepsilon}}\Big]\\
&+C_{32}e^{-\frac{\mu}{\varepsilon}(c+\varepsilon v)}\cdot(c+\varepsilon v)^{-\frac{1}{2}+\frac{\lambda}{2\varepsilon}}\cdot {\rm BesselK}\Big[-1+\frac{\lambda}{\varepsilon},2\sqrt{\frac{(u+\frac{c}{\varepsilon})\lambda\mu}{\varepsilon}}\Big],
\end{align*}
where $C_{31}$ and $C_{32}$ are some constants and BesselI and BesselK are modified Bessel functions. In this case $n=-1+\frac{\lambda}{\varepsilon}$ has to be restricted to an integer. This yields
\begin{align*}
s_{31}(u)&=C_{31}\int_{u}^{\infty}e^{-\frac{\mu}{\varepsilon}(c+\varepsilon v)}\cdot(c+\varepsilon v)^{-\frac{1}{2}+\frac{\lambda}{2\varepsilon}}\cdot {\rm BesselI}\Big[-1+\frac{\lambda}{\varepsilon},2\sqrt{\frac{(v+\frac{c}{\varepsilon})\lambda\mu}{\varepsilon}}\Big]dv,\\
s_{32}(u)&=C_{32}\int_{u}^{\infty}e^{-\frac{\mu}{\varepsilon}(c+\varepsilon v)}\cdot(c+\varepsilon v)^{-\frac{1}{2}+\frac{\lambda}{2\varepsilon}}\cdot {\rm BesselK}\Big[-1+\frac{\lambda}{\varepsilon},2\sqrt{\frac{(v+\frac{c}{\varepsilon})\lambda\mu}{\varepsilon}}\Big]dv.
\end{align*}
\\
Since
\begin{equation*}
{\rm I}_v(z)\sim \frac{e^z}{\sqrt{2\pi z}}\text{ and }{\rm K}_v(z)\sim \sqrt{\frac{\pi}{2z}}e^{-z}\quad\text{as }z\rightarrow\infty
\end{equation*}
(see \citet[p. 377]{ashandbook}), we have that, for $u\rightarrow\infty$,
\begin{align*}
\psi_{l,3}(u)&\sim\frac{C_{31}}{\sqrt{2\pi\sqrt{\lambda\mu}}}e^{-\frac{\mu c}{\varepsilon}}\int_{u}^{\infty}e^{-\mu y+\frac{2}{\varepsilon}\sqrt{\lambda\mu(c+\varepsilon y)}}\cdot(c+\varepsilon y)^{-\frac{3}{4}+\frac{\lambda}{2\varepsilon}}dy\\
&+C_{32}\sqrt{\frac{\pi}{2\sqrt{\lambda\mu}}}e^{-\frac{\mu c}{\varepsilon}}\int_{u}^{\infty}e^{-\mu y-\frac{2}{\varepsilon}\sqrt{\lambda\mu(c+\varepsilon y)}}\cdot(c+\varepsilon y)^{-\frac{3}{4}+\frac{\lambda}{2\varepsilon}}dy
\end{align*}
which is consistent with \eqref{asympdwa} and hence with Theorem \ref{thm:2}.

\section{Asymptotic Analysis - comparison results}
\label{asymptotic}
\subsection{Exp$(\lambda)$ distributed interarrival times with Exp$(\mu)$ distributed claims sizes}
In this case, for the linear premium we have the asymptotic result (\ref{asymptotic1}), that is,
\begin{equation*}
\psi(u)\sim\frac{\mu}{\lambda c^{\lambda/\varepsilon}}Ce^{-\mu u}(c+\varepsilon u)^{\frac{\lambda}{\varepsilon}-1},
\end{equation*}
where $C$ is some constant. For a constant premium $c$, the result for the ruin probability is
\begin{equation*}
\psi_{c,1}(u)=\frac{\lambda}{c\mu}e^{-(\mu-\frac{\lambda}{c})u},\quad\mbox{for any $u\geqslant0$}.
\end{equation*}
Thus, we have
\begin{equation*}
\psi_{l,1}(u)\sim C\cdot\psi_{c,2}(u)\cdot e^{-\frac{\lambda}{c}u}\cdot(c+\varepsilon u)^{\frac{\lambda}{\varepsilon}-1},\quad \mbox{as $u\rightarrow\infty$}.
\end{equation*}

\subsection{Erlang$(2,\lambda)$ distributed interarrival times with Exp$(\mu)$ distributed claims sizes}
Recall the asymptotic result (\ref{l2}) for risk models with linear premiums in this case
\begin{align*}
\psi_{l,2}(u)&\sim
C_1\cdot{\Big(\frac{\mu}{\varepsilon}\Big)}^{-1-\frac{\varepsilon+2\lambda+\varepsilon\sqrt{1+\frac{4\lambda}{\varepsilon}}}{2\varepsilon}}
\cdot\frac{1}{\varepsilon}\Big(\frac{e^{-\mu u}}{c+\varepsilon u}-\frac{\mu}{\varepsilon}e^{\frac{\mu}{\varepsilon}c}\Gamma\left(\frac{\mu}{\varepsilon}(c+\varepsilon u),0\right)\Big),\quad \mbox{as $u\rightarrow\infty$}
\end{align*}\\
and the explicit result for risk models with constant premiums gives
\begin{equation*}
\psi_{c,2}(u)=C_{2}e^{-\frac{c\mu-2\lambda+\sqrt{c^2\mu^2+4c\lambda\mu}}{2c} u},\quad u\geqslant0
\end{equation*}
(see \citet{dickson1998ruin,dicksonhipperlang2}). Taking the limit and applying L'H$\hat{\rm o}$pital's rule, the ratio between $\psi_{l,2}(u)$ and $\psi_{c,2}(u)$
behaves asymptotically as
\begin{equation*}
\frac{\psi_{l,2}(u)}{\psi_{c,2}(u)}\sim C_3 e^{-\frac{c\mu+2\lambda-\sqrt{c^2\mu^2+4c\lambda\mu}}{2c} u}(c+\varepsilon v)^{-2},\quad u\rightarrow\infty
\end{equation*}
where $C_3$ is some constant. Hence $\frac{\psi_{l,2}(u)}{\psi_{c,2}(u)}$
tends to zero as $u$ tends to infinity.\\

This means that as the initial surplus $u$ increases, one has more premium income for risk models with linear premiums, thus the ruin probability $\psi_{l,2}(u)$ for risk models with linear premiums decays to zero exponentially faster than
the ruin probability $\psi_{c,2}(u)$ for constant premiums risk models. As expected, this means that risk models with linear premiums are less risky than the constant premiums ones.

\subsection{Exp$(\lambda)$ distributed interarrival times with Erlang$(2,\mu)$ distributed claims sizes}
We start from recalling the asymptotic result (\ref{asympdwa}) for risk models with linear premiums
\begin{align*}
\psi_{l,3}(u)\sim&
C_1\int_u^\infty e^{-\mu y+\frac{2}{\varepsilon}\left(\sqrt{\lambda\mu(c+\varepsilon y)}-\sqrt{\lambda\mu c}\right)}\left(\frac{c+\varepsilon y}{c}\right)^{-\frac{3}{4}+\frac{\lambda}{2\varepsilon}}dy\nonumber\\
&+
C_2\int_u^\infty e^{-\mu y-\frac{2}{\varepsilon}\left(\sqrt{\lambda\mu(c+\varepsilon y)}-\sqrt{\lambda\mu c}\right)}\left(\frac{c+\varepsilon y}{c}\right)^{-\frac{3}{4}+\frac{\lambda}{2\varepsilon}}dy,\quad u\rightarrow\infty
\label{asympdwa2}
\end{align*}
where $C_1$ and $C_2$ are constants.
%
The explicit result for constant premiums case can also be derived from ordinary differential equation as
\begin{equation*}
\psi_{c,3}(u)=C_{3}e^{\sigma_{31}u}+C_{4}e^{\sigma_{32}u},\quad u\geqslant0,
\end{equation*}
where $\sigma_{31,32}=-\frac{2c\mu-\lambda\pm\sqrt{\lambda^2+4c\lambda\mu}}{2c}<0$ and $C_{3}$ and $C_{4}$ are some constants, see \cite{ligarridoerlangn,bergel2015further} for details.\\

Again, taking the limit and applying L'H$\hat{\rm o}$pital's rule, we can conclude that
\begin{equation*}
\lim\limits_{u\rightarrow\infty}\frac{\psi_{l,3}(u)}{\psi_{c,3}(u)}
=0.
\end{equation*}
Thus, as the initial surplus $u$ increases, the ruin probability $\psi_{l,3}(u)$ for risk models with linear premiums decreases to zero faster than
the ruin probability $\psi_{c,3}(u)$ for constant premiums. Again,  this means that risk models with constant premiums are more risky than linear premiums ones, as expected,
thus there is gain in terms of solvency when binding premium to reserves.

\section{Conclusion}
\label{conclusion}
It is much easier to calculate the ruin probabilities for risk models with constant premiums, and explicit results for constant cases abound in risk theory literature, but the risk models with surplus-dependent premiums are more applicable in real life. For these complex cases, we have results in terms of confluent hypergeometric function and modified Bessel function at most, or only asymptotic results, from which one can make inferences.

\bibliographystyle{chicago}
\bibliography{bib}

\end{document}